\documentclass[11pt,reqno]{amsart}

\numberwithin{equation}{section}
\usepackage[all,cmtip]{xy}
\usepackage{amssymb,rotating,amsmath,amsfonts,amsthm,color,graphics,epsfig,bbm,bm,manfnt,a4wide,psfrag}

\newcommand{\calA}{\mathcal{A}}

\newcommand{\mA}{\mathbb{A}}

\newcommand{\mC}{\mathbb{C}}
\newcommand{\mD}{\mathbb{D}}

\newcommand{\mF}{\mathbb{F}}

\newcommand{\mR}{\mathbb{R}}
\newcommand{\mS}{\mathbb{S}}
\newcommand{\mT}{\mathbb{T}}

\newcommand{\mZ}{\mathbb{Z}}

\newcommand{\bbc}{\mathbf{c}}
\newcommand{\bbd}{\mathbf{d}}

\newcommand{\bbn}{\mathbf{n}}

\newcommand{\bbp}{\mathbf{p}}

\newcommand{\bbx}{\mathbf{x}}
\newcommand{\bby}{\mathbf{y}}

\newcommand{\bbH}{\mathbf{H}}

\newcommand{\inv}{{\textrm{inv }}}
\newcommand{\nm}{\,\rule[-.6ex]{.13em}{2.3ex}\,}

\newtheorem{theorem}{Theorem}[section]
\newtheorem{lemma}[theorem]{Lemma}

\newtheorem{proposition}[theorem]{Proposition}

\theoremstyle{definition}

\theoremstyle{definition}
\newtheorem{definition}[theorem]{Definition}

\theoremstyle{definition}

\begin{document}

\keywords{$\nu$-metric, robust control, Banach algebras}

\subjclass{Primary 93B36; Secondary 93D15, 46J15}

\title[Refinement of the generalized chordal distance]{A 
refinement of the generalized chordal distance}

\author{Amol Sasane}
\address{Department of Mathematics, Faculty of Science, 
    Lund University, Sweden.}
\email{amol.sasane@math.lu.se}

\begin{abstract}
For single input single output systems, we give a refinement $d_{c,r}$ of the 
generalized chordal metric $d_c$ introduced in \cite{Sasb}. Our metric $d_{c,r}$ 
is given in terms of coprime factorizations, but it coincides with the extension of
Vinnicombe's $\nu$-metric given in Ball and Sasane \cite{BalSas} if
the coprime factorizations happen to be  normalized.  The
advantage of the metric $d_{c,r}$ introduced in this article is its easy
computability (since it relies only on coprime factorizations, and does not require 
 {\em normalized} coprime factorizations). We also give concrete 
formulations of our abstract metric for standard classes of stable transfer functions.
\end{abstract}

\maketitle

\section{Introduction}

The Vinnicombe $\nu$-metric introduced in \cite{Vin} and its abstract
version given in \cite{BalSas} both rely on finding {\em normalized}
coprime factorizations. This is a troublesome aspect of the theory,
since
\begin{enumerate}
\item Although merely coprime factorizations may exist, a {\em
    normalized} coprime factorization may fail to exist: for example
  in the article \cite{Tre} by Sergei Treil, it was shown that the set of plants in the field of
  fractions of the disk algebra possessing normalized coprime
  factorizations is strictly contained in the set of plants possessing
  coprime factorizations.
\item Even if they exist, normalized coprime factorizations might be
  impossible to find using a constructive procedure: for example, in
  the paper \cite{PS} by Jonathan Partington and Gregory Sankaran, 
  it is shown that in the case of delay systems, in general
  the relevant spectral factorizations for finding normalized coprime 
  factorizations cannot be determined by solving
  any finite system of polynomial equations over the field
  $\mR(s,e^{−s})$.
\end{enumerate}
The goal in this paper is to show that this problematic feature of the
$\nu$-metric can be eliminated, at least in the case of single input
single output systems, by redefining the $\nu$-metric, which relies
only on coprime factorizations, rather than {\em normalized} coprime
factorizations.  The starting point is the generalized chordal
distance $d_c$ introduced in \cite{Sasb} given in terms of coprime
factorizations, and then considering a refinement $d_{c,r}$ of this
chordal distance akin to the $\nu$-metric $d_\nu$. It turns out that
the metric $d_{c,r}$ coincides with the $\nu$-metric if one has
normalized coprime factorizations at hand. But since our metric is in
fact only defined using coprime factorizations, the burden of working
with normalized coprime factorizations is completely eliminated. Our
main results are then that $d_{c,r}$ defines a metric on the set of
all elements admitting a coprime factorization, and stabilizability is
a robust property of the plant. The precise statements of the results
are given in Theorems~\ref{thm_main_1} and \ref{thm_main_2}, after the
notation has been introduced in the next section. 

\section{Background, known results and related literature} 

Let us recall that in the  ``factorization approach'' to linear control theory, 
one starts with an integral domain $R$, which is thought to constitute 
the set of all transfer functions of {\em stable} 
linear systems. Transfer functions of linear 
systems which are not necessarily stable are then taken to be 
elements of the field of fractions $\mF(R)$ over $R$.  Based on algebraic 
factorizations of the plant transfer function, control theoretic problems can 
then be posed and solved. We refer the uninitiated reader to the monograph by \cite{Vid}. 
This factorization approach to linear control theory has been 
resurrected and extended in the articles \cite{Qua}, \cite{Quab}, \cite{Mik} 
and the references therein. In particular, in \cite{Qua}, \cite{Quab}, 
a theory of solving the stabilization problem (recalled below) is developed, 
which does not rely on the existence of coprime factorizations, 
but instead proceeds under weaker notions of coprimeness. 
This theme is also present in \cite{Mik}, with further emphasis on 
obtaining stabilizing controllers that are ``real symmetric'' 
(and hence physically implementable) given real symmetric data. Our concern 
in this article is not with the stabilization problem, 
but rather the {\em robust} stabilization problem. 
Before recalling the robust stabilization problem, 
let us first remind ourselves of the stabilization problem. 

The {\em stabilization problem} is the following: 
Given $\bbp\in\mF(R)$ (and unstable plant), find 
$\bbc\in \mF(R)$ (a stabilizing controller) such that 
the ``interconnection'' of $\bbp$ and $\bbc$ is stable, that is,  
the closed loop transfer function is stable:
 $$
 \left[\begin{array}{cc} \displaystyle\frac{-\bbp  \bbc}{1-\bbp\bbc} & \displaystyle\frac{\bbp}{1-\bbp \bbc}\\
       \displaystyle\frac{-\bbc}{1-\bbp\bbc} & \displaystyle\frac{1}{1-\bbp\bbc}
      \end{array}\right]\in R^{2\times 2}.
$$
It is well known that this problem can be solved if $\bbp$ has a ``coprime factorization'' \cite{Vid}. 

In reality, the plant transfer function is computed from the differential equation model of the 
situation at hand, which in turn is obtained from a modelling procedure involving idealizations, 
simplifying assumptions, approximations, and measurement of parameters. This means that 
the plant transfer function is not known precisely, but serves only as an approximation of reality. 
Hence engineers imagine that all we know is a ``nominal'' plant transfer function $\bbp_0$, and the reality 
might well be a perturbation $\bbp$ of this nominal transfer function $\bbp_0$. So one wishes the stabilizing controller 
$\bbc$ found for the nominal $\bbp_0$ to stabilize not only $\bbp_0$, but also all plants $\bbp$ close enough to $\bbp_0$, 
and we would also like to be able to compute the radius of this 
neighbourhood (so that the engineers are aware of the parametric uncertainity tolerated). 
\begin{figure}[h]
   \center
   \psfrag{p}[c][c]{$\bbp_0$}   
   \psfrag{c}[c][c]{$\bbc$}   
   \includegraphics[width=6 cm]{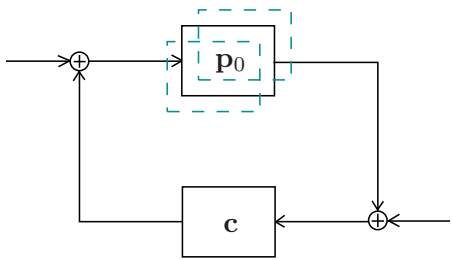}
   \caption{Robust stabilization of an uncertain plant.}
   \label{interconnection}
\end{figure}   
  The question of what should be an appropriate notion of distance one should use to measure 
closeness of unstable plants thus arises naturally. Based on the goals described in the above paragraph, 
it is natural to demand  a  metric $d$ on the set of stabilizable plants  such that  
\begin{enumerate}
 \item $d$ is easily computable, and 
 \item $d$ has the following  property with respect to the Stabilization Problem: Stabilizability becomes 
 a robust property of stabilizable plants in this metric. In other words, if $\bbp_0$ is stabilized by a controller $\bbc$, 
 then there exists an $r>0$ such that every $\bbp$ in the ball $B(\bbp_0,r):=\{\bbp:d(p,p_0)<r\}$ is stabilized by $\bbp$, 
 and this $r$ should be easily computable too. 
\end{enumerate}
Such a metric was introduced by Vinnicombe in \cite{Vin} when $R=RH^\infty$ (rational functions without poles 
in the right half plane or the imaginary axis), and was called the $\nu$-metric in \cite{Vin}. The definition of the $\nu$-metric 
 was given in terms of {\em normalized} coprime factorizations. An abstract extension of this metric was given in \cite{BalSas}  
 in order to cover the case when $R\neq RH^\infty$, but rather $R$ is the class of stable transfer functions 
 of infinite dimensional systems, for example, when $R$ is the Callier-Desoer algebra ($\calA^+$ defined on 
 page \pageref{def_Callier_Desoer} in the present article). The starting point in \cite{BalSas} 
 was abstract, where it was assumed that $R$  satisfies some mild assumptions, and then an abstract $\nu$-metric was defined on those elements of $\mF(R)$ that possess 
 a {\em normalized} coprime factorization. Specific examples of $R$ satisfying the abstract assumptions and the resulting particular examples 
 of the $\nu$-metric were given in \cite{BalSas}. Notably, that the case when $R=\calA^+$ is covered by the abstract set up was given in  
 \cite{BalSas}, but not the case when $R=H^\infty$ (bounded and holomorphic functions in the unit disc with center $0$ in $\mC$). \label{def_hardy} Subsequently, in the auxiliary papers \cite{BalSasH}, \cite{Sas} and \cite{FreSas}, the abstract 
 $\nu$-metric introduced in \cite{BalSas} was {\em applied} in the particular case when $R=H^\infty$.  \cite{FreSas} 
 was essentially a new reformulation of \cite{Sas} (where the case $R=H^\infty$ was covered first), 
 while \cite{BalSasH} covered only the case of $R=QA$ (quasi analytic functions).
 In {\em all} of these, 
 just like in \cite{BalSas}, the $\nu$-metric relied on the existence of {\em normalized} coprime factorizations. 
 As pointed out in the introduction, existence of a coprime factorization of a plant does not 
 automatically imply the existence of a normalized one, and even if a normalized one is known to exist, there may not 
 be a computational procedure to find it. On the other hand, the problem of finding coprime factorizations is much more tractable 
 than finding normalized coprime ones, and 
 so it is desirable to find an extension of Vinnicombe's  metric which is given on the set of all plants that possess a coprime factorization. 
 It is this issue that is addressed in the present paper.  
 Hence the results in the present work are not a consequence of these previous works. See the schematic diagram on page \pageref{heirachy}.
 \begin{figure}[h]
   \center
   \psfrag{A}[c][c]{$\textrm{Present article: } \begin{array}{cc}\textrm{Refinement of the abstract chordal distance}\\
                      \textrm{(distance defined using coprime factorizations)}
                     \end{array}
$}   
   \psfrag{B}[c][c]{$\textrm{\cite{BalSas}: }\begin{array}{cc}\textrm{Abstract } \nu\textrm{-metric for single input single output  systems}\\
                      \textrm{(distance defined using normalized coprime factorizations)}
                     \end{array} $}   
           \psfrag{C}[c][c]{$\begin{array}{llll} \textrm{\cite{BalSas} } R=RH^\infty \textrm{(Vinnicombe }\nu\textrm{-metric)}\\
           \textrm{\cite{BalSas} } R=\calA^+\\
           \textrm{\cite{BalSasH} } R=QA\\
           \textrm{\cite{Sas}}\equiv \textrm{\cite{FreSas} }R=H^\infty
                     \end{array} $}              
   \includegraphics[width=13.2cm]{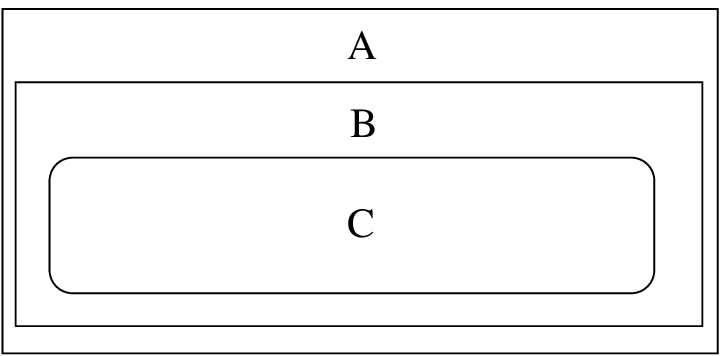}
\label{heirachy}
\end{figure}

 Finally, we remark that in the previous work \cite{Sasb}, we showed robustness of the property of {\em strong} stabilizability 
 of the ``generalized chordal distance'' in a somewhat different setting. The point of the present work is that by considering 
 a {\em refinement} of the generalized chordal distance from \cite{Sasb}, we can show that the robustness of the property of 
 {\em stabilizability}.  We remark that in general rings $R$ (for example in $R=\calA^+$), the set of stabilizable plants is strictly bigger than 
 the set of strongly stabilizable plants. 
 
 \section{Notation} 
 
 For the convenience of the reader, we have included a table here which
shows the page numbers of the places where the corresponding notation is
first defined.

\medskip

\begin{center}
\begin{tabular}{|c||l|} \hline
  Notation              & Page number \\ \hline \hline
  $R$                   & (A1) on page \pageref{def_R} \\\hline
  $S$, $I$              & (A2) on page \pageref{def_S} \\\hline
  $\cdot^*$             & involution; (A2) on page \pageref{def_I} and page \pageref{involution} \\\hline
  $\iota$, $\circ$, $G$ & (A3) on page \pageref{def_iota} \\\hline
  $\mF(R)$              & field of fractions over $R$; page \pageref{def_F(R)} \\\hline
  $\mS$                 & set of plants having a coprime factorization; page   \pageref{def_mS}\\\hline
  $\mS_n$               & set of plants having a normalized coprime factorization; page   \pageref{def_mSn}\\\hline
  $\widehat{\;\cdot\;}$ & Gelfand/Laplace transform, page \pageref{def_gelfand} and page \pageref{def_Laplace}\\\hline
  $M(S)$                & maximal ideal space of $S$; page \pageref{def_maximal}\\ \hline 
  $\|\cdot\|_S=\|\cdot\|_{S,\infty}$ & norm in the $C^*$ algebra $S$; page \pageref{def_norm}\\ \hline
  $\nm\;\cdot \;\nm$    & induced operator norm from $\mC^p$ to $\mC^m$; page   \pageref{def_op_norm}\\ \hline
  $\kappa_{\bbp_1, \bbp_2}$ & chordal distance function between $\bbp_1, \bbp_2$; page \pageref{def_chordal_dis}\\ \hline
  $d_{c,r}$             & refinement of  the  chordal distance; page \pageref{eq_nu_metric} \\ \hline 
  $\mu_{\bbp,\bbc}$     & stability margin of $(\bbp, \bbc)$; page \pageref{def_stab_margin}\\ \hline 
  $\mD,\mT, \overline{\mD},  \mC_{\scriptscriptstyle >0}$ $\phantom{\displaystyle \int}$ & certain subsets of $\mC$; page \pageref{def_subsets_of_C} \\ \hline 
  $H^\infty$            & the Hardy algebra; page \pageref{def_hardy}\\ \hline 
  $RH^\infty$           & set of rational elements of $H^\infty$; page \pageref{def_classical_algebras} \\ \hline 
  $W^+(\mD)$            & the Wiener algebra; page \pageref{def_classical_algebras} \\ \hline 
  $A(\mD)$              & the disc algebra; page \pageref{def_classical_algebras} \\ \hline 
  $\widehat{L^1[0,\infty)}+\mC$ $\phantom{\displaystyle \int}$& the Wiener-Laplace algebra; page \pageref{def_classical_algebras} \\ \hline 
  $C(\mT)$              & set of continuous functions on $\mT$; page \pageref{def_C(T)} \\ \hline 
  $w$                   & page \pageref{def_winding_number} and page \pageref{def_w} \\ \hline 
  $\calA^+$             & page \pageref{def_Callier_Desoer}\\ \hline 
  $C_0$                 & set of continuous functions on $\mR$ vanishing at $\pm\infty$; page \pageref{def_C_0}\\ \hline 
  $AP$                  & set of almost periodic functions; page \pageref{def_C_0}\\ \hline 
  $w_{av}$              & average winding number; page \pageref{def_average_winding}\\ \hline 
  $\mA_r, \mA_0$        & annulus; page \pageref{def_annulus}\\ \hline          
  $\pi_r^R$             & restriction; page \pageref{def_pi_r^R}\\ \hline 
  $\pi_r$               & restriction; page \pageref{def_pi_r}\\ \hline 
  $C_b(\mA_r)$, $\displaystyle \prod_{r\in (0,1)}C_b(\mA_r)$         & page \pageref{def_C_b} \\ \hline 
  $\calA$, $N$               & page \pageref{def_calA}\\ \hline
  $\beta \mA_0$, $C(\beta \mA_0\setminus \mA_0)$       & page \pageref{def_stone} \\ \hline
\end{tabular}

\end{center}

 \goodbreak 
 
\section{Setup and preliminaries}

\noindent Our setup is the following:
\begin{itemize}
\item[(A1)] $R$ is commutative integral domain with identity.\label{def_R}
\item[(A2)] $S$ is a commutative $C^\ast$-algebra such that $R \subset
  S$, that is there is an injective ring homomorphism $I:R\rightarrow
  S$.\label{def_S}\label{def_I}
\item[(A3)] Let $\inv S$ to denote the invertible elements of $S$.
  There exists a map $\iota: \inv S \rightarrow G$, where $(G,+)$ is
  an Abelian group with identity denoted by $\circ$, and $\iota$
  satisfies\label{def_iota} \label{def_circ} \label{def_G}
\begin{itemize}
\item[(I1)] $\iota(ab)= \iota (a) +\iota(b)$ ($a,b \in \inv S$).
\item[(I2)] $\iota(a^*)=-\iota(a)$ ($a\in \inv S$).
\item[(I3)] $\iota$ is locally constant, that is, $\iota$ is
  continuous when $G$ has the discrete topology.
\item[(I4)] $\bbx\in R \cap (\inv S)$ is invertible as an element of
  $R$ if and only if $\iota(\bbx)=\circ$.
\item[(I5)] If $\bbx\in S$ and $\bbx>0$, then $\iota(\bbx)=\circ$.
\end{itemize}
\end{itemize}

\medskip 

\noindent (A3) allows identification of elements of $R$ with elements
of $S$. So in the sequel, if $\bbx$ is an element of $R$, we will
simply write $\bbx$ (an element of $S$!) instead of $I(\bbx)$.  A
consequence of (I3) is the following homotopic invariance of the index
\cite[Proposition~2.1]{BalSas}, which we will use in the sequel. 
(I4) is to be thought of as an abstract Nyquist criterion. We
also remark that there are salient differences in the setup in
\cite{BalSas} versus the above. In particular, we take $S$ to be a
$C^*$-algebra, whereas in \cite{BalSas}, $S$ was a Banach algebra, and
also the condition (I5), which is critically used in our proofs, was
missing in the framework of \cite{BalSas}.

\begin{proposition}
\label{prop_hom_inv}
If $H:[0,1] \rightarrow \textrm{\em inv } S$ is continuous, then $
\iota (H(0))=\iota(H(1))$.
\end{proposition}
\begin{proof} This follows from  (I3); see \cite[Proposition~2.1]{BalSas}.
 \end{proof}

\noindent We also introduce the following standard notation and
terminology:

\medskip 

\noindent $\mF(R)$ denotes the field of fractions of $R$. \label{def_F(R)} Given $\bbp \in \mF(R)$, a factorization
$$
\bbp=\frac{\bbn}{\bbd},
$$
where $\bbn \in R$, $\bbd\in R\setminus\{ 0\}$, is called a {\em coprime factorization of}
$\bbp$ if there exist $\bbx, \bby\in R$ such that
$$
\bbn \bbx+\bbd \bby=1.
$$
We denote by $\mS$ the set of all elements in $\mF(R)$ that possess a
coprime factorization.\label{def_mS}

The maximal ideal space of $S$ is denoted by $M(S)$. \label{def_maximal}  If $\bbx\in S$,
then we denote by $\widehat{\bbx}$ \label{def_gelfand}the Gelfand transform of
$\bbx$. Also, by the Gelfand-Naimark Theorem \cite[Theorem~11.18]{Rud},\label{def_norm}
$$
\|\bbx\|_S
=
\|\widehat{\bbx}\|_{S,\infty}
=
\max_{\varphi \in M(S)}|\widehat{\bbx}(\varphi)|.
$$
Let  $\bbp_1,\bbp_2 \in \mS(R)$ have coprime factorizations
$$
\bbp_1=\displaystyle \frac{\bbn_1}{\bbd_1}
\textrm{ and }
\bbp_2=\displaystyle \frac{\bbn_2}{\bbd_2}.
$$
Then the {\em chordal pointwise distance function} \label{def_chordal_dis}
$\kappa_{\bbp_1,\bbp_2}$ is 
$$
\kappa_{\bbp_1,\bbp_2}(\varphi)
:=
\frac{|\widehat{\bbn_1}(\varphi)\widehat{\bbd_2}(\varphi)
-\widehat{\bbn_2}(\varphi)\widehat{\bbd_1}(\varphi)|}{
\sqrt{|\widehat{\bbn_1}(\varphi)|^2+|\widehat{\bbd_1}(\varphi)|^2} 
\sqrt{|\widehat{\bbn_2}(\varphi)|^2+|\widehat{\bbd_2}(\varphi)|^2}},
\quad \varphi \in M(S).
$$
The function $\kappa$ given by the above expression is well-defined,
that is, it does not depend on the choice of coprime factorizations
for each of the plants; see \cite{Sasb}.

\begin{definition}
For $\bbp_1, \bbp_2 \in \mS$, with coprime factorizations
$$
\bbp_1= \frac{\bbn_{1}}{\bbd_{1}},\quad \bbp_2= \frac{\bbn_{2}}{\bbd_{2}},
$$
we define
\begin{equation}
\label{eq_nu_metric}
d_{c,r} (\bbp_1,\bbp_2 ):=\left\{
\begin{array}{ll}
  \|\kappa_{\bbp_1,\bbp_2}\|_{S,\infty}  &
  \textrm{if } \bbn_1^* \bbn_2+\bbd_1^* \bbd_2 \in \inv S \textrm{ and }
  \iota ( \bbn_1^* \bbn_2+\bbd_1^* \bbd_2)=\circ, \\
  1 & \textrm{otherwise}. \end{array}
\right.
\end{equation}
\end{definition}

\medskip 

\noindent $d_{c,r}$ given by \eqref{eq_nu_metric} is well-defined, and
it is bounded above by $1$. It is also clear that if the two plants
have {\em normalized} coprime factorizations, that is, if also
$$
\bbn_{1}^*\bbn_1+\bbd_{1}^*\bbd_1=1=\bbn_{2}^*\bbn_2+\bbd_{2}^*\bbd_2,
$$
then it is clear that $d_{c,r}(\bbp_1,\bbp_2)=d_\nu(\bbp_1,\bbp_2)$, where 
$d_\nu$ is the extension of the $\nu$-metric defined in \cite{BalSas}. We note however, 
that $d_\nu$ was defined on the set $\mS_n$ \label{def_mSn}
of elements in the field of fractions possessing a normalized coprime factorization, and for some 
rings $R$ (for example $R=A(\mD)$, the disk algebra; see \cite{Tre}), it may be the case that $\mS_n\subsetneq \mS$. 

Our first main result is the following:

\begin{theorem}
\label{thm_main_1}
$d_{c,r}$ is a metric on $\mS$.
\end{theorem}

We will also show that stabilizability is a robust property of the
plant and give the quantitative version of this in
Theorem~\ref{thm_main_2} below.  But first, we have the following
definition of the stability margin.

\begin{definition}
Given $\bbp\in \mF(R)$, we say that $\bbc\in \mF(R)$ {\em
    stabilizes} $\bbp$ if
$$
\bbH(\bbp,\bbc)
:=\left[\begin{array}{cccc} 
\displaystyle \frac{-\bbp \bbc}{1-\bbp \bbc} &  
\displaystyle \frac{\bbp}{1-\bbp \bbc} \\
\displaystyle \frac{- \bbc\phantom{^{\displaystyle f}} }{1-\bbp \bbc} & 
\displaystyle \frac{1}{1-\bbp \bbc} 
\end{array}\right]\in R^{2\times 2}.
$$
We define the {\em stability margin} $\mu_{\bbp,\bbc}$ of
$\bbp,\bbc\in \mF(R)$ by\label{def_stab_margin}
$$
\mu_{\bbp,\bbc}
=
\left\{\begin{array}{cl} 
\displaystyle \frac{1}{\|\bbH (\bbp,\bbc)\|_{S,\infty}} & 
\textrm{if } \bbc\textrm{ stabilizes } \bbp,\\
0\phantom{\displaystyle\frac{a}{b}} & \textrm{otherwise}.
\end{array}\right.
$$
Here for a matrix $M\in S^{p\times m}$ with entries from $S$, we set
$$
\|M\|_{S,\infty}=\sup_{\varphi \in M(S)} \nm \widehat{M}(\varphi)\nm,
$$
where $\widehat{M}$ denotes the matrix obtained by taking entrywise
Gelfand transforms, and $\nm\cdot \nm$  \label{def_op_norm} denotes the induced operator
norm when $\mC^p, \mC^m$ are equipped with the usual Euclidean norms.
\end{definition}

\begin{theorem}
\label{thm_main_2}
If $\bbp,\bbp_0,\bbc\in \mS$, then $\mu_{\bbp,\bbc}\geq
\mu_{\bbp_0,\bbc}-d_{c,r}(\bbp,\bbp_0)$.
\end{theorem}

In order to show Theorem~\ref{thm_main_2}, the following fact will be
useful.

\begin{lemma}
If $\bbp,\bbc\in \mS$ have coprime factorizations $\bbp=
\displaystyle \frac{\bbn_p}{\bbd_p}$ and $\bbc=\displaystyle  
\frac{\bbn_c}{\bbd_c}$, then 
$$ 
\mu_{\bbp,\bbc}=\displaystyle
  \inf_{\varphi \in M(S)}
  \frac{|\widehat{\bbn}_p(\varphi)\widehat{\bbn}_c(\varphi)-
    \widehat{\bbd}_p(\varphi)\widehat{\bbd}_c(\varphi)|}{
    \sqrt{|\widehat{\bbn}_p(\varphi)|^2+|\widehat{\bbd}_p(\varphi)|^2
    }\sqrt{
      |\widehat{\bbn}_c(\varphi)|^2+|\widehat{\bbd}_c(\varphi)|^2}}.
$$
\end{lemma}
\begin{proof} The claim follows from the observations that 
$$
\bbH(\bbp,\bbc)=
\frac{1}{\bbd_p\bbd_c-\bbn_p \bbn_c}
\left[\begin{array}{cc} 
\bbn_p \\
\bbd_p 
\end{array}\right]
\left[\begin{array}{cc} -\bbn_c & \bbd_c
\end{array}\right]
$$
and
$$
\nm 
 \left[\begin{array}{cc} 
 \widehat{\bbn}_p (\varphi)\\
 \widehat{\bbd}_p (\varphi)
 \end{array}\right]
 \left[\begin{array}{cc} 
 -\widehat{\bbn}_c (\varphi)&
 \widehat{\bbd}_c (\varphi)
 \end{array}\right]
 \nm =
 \sqrt{|\widehat{\bbn}_p(\varphi)|^2+|\widehat{\bbd}_p(\varphi)|^2 }\sqrt{
 |\widehat{\bbn}_c(\varphi)|^2+|\widehat{\bbd}_c(\varphi)|^2},
$$
for $\varphi \in M(S)$.
\end{proof}

It follows from the above that $\mu_{\bbp,\bbc}\leq 1$. Also,
$\mu_{\bbp,\bbc}=\mu_{\bbc,\bbp}$.

\begin{proposition}
\label{proposition_stabilization_index}
Let $\bbp,\bbc\in \mS$ have coprime factorizations 
$$
\bbp=\displaystyle \frac{\bbn_p}{\bbd_p}\textrm{ and } 
\bbc=\displaystyle \frac{\bbn_c}{\bbd_c}.
$$
Then the following are equivalent:
\begin{enumerate}
\item $\bbc$ stabilizes $\bbp$.
\item $\bbn_p \bbn_c -\bbd_p\bbd_c\in \textrm{\em inv }S$ and
  $\iota(\bbn_p \bbn_c -\bbd_p\bbd_c)=\circ$.
\end{enumerate}
\end{proposition}
\begin{proof} Suppose that $\bbc$ stabilizes $\bbp$. Then it 
can be seen that $\bbn_p\bbn_c-\bbd_p\bbd_c$ belongs to 
$\inv R$.  It follows from (I4) that $\bbn_p \bbn_c -\bbd_p\bbd_c 
\in \inv S$ and that $\iota(\bbn_p \bbn_c -\bbd_p\bbd_c)=\circ$.

On the other hand, if $\bbn_p \bbn_c -\bbd_p\bbd_c\in \inv S$ and
$\iota(\bbn_p \bbn_c -\bbd_p\bbd_c)=\circ$, then again by (I4), we
obtain that $\bbn_p\bbn_c-\bbd_p\bbd_c\in \inv R$. It follows from
here that $\bbH(\bbp,\bbc)$ has each entry in $R$.
\end{proof}

The following elementary fact will be used often in our proofs. 
\begin{lemma}
\label{lemma_cn_fact}
For complex numbers $a,b,\alpha,\beta$, such that $|a|^2+|b|^2>0$ and
$|\alpha|^2+|\beta|^2>0$ there holds that
$$
1-\frac{|a\beta-b\alpha|^2}{(|a|^2+|b|^2)(|\alpha|^2+|\beta|^2)}
=
\frac{|a\overline{\alpha}+b\overline{\beta}|^2}{(|a|^2+|b|^2)(|\alpha|^2+|\beta|^2)}.
$$
\end{lemma}

The proof of Theorems~\ref{thm_main_1} and \ref{thm_main_2} are analogous to the 
 proofs given in \cite{Vin} (in the case when $R$ is the ring of rational functions bounded in the right half  complex plane) 
 and the proofs given in \cite{BalSas} (in an abstract setting). However, Theorems~\ref{thm_main_1} and \ref{thm_main_2} 
 are not automatic, and do not follow from \cite{BalSas}, since we consider a more general setting 
 than the one considered in \cite{BalSas} (merely coprime factorizations versus normalized coprime factorizations).  

\section{$d_{c,r}$ is a metric}
\label{section_metric}

\begin{proof}[Proof of Theorem~\ref{thm_main_1}] 
The fact that $d_{c,r}$ is nonnegative is trivial. Also, if 
$d_{c,r}(\bbp_1,\bbp_2)=0$, then it is clear that
$\bbp_1=\bbp_2$. Symmetry is also easy to check. We only check the
triangle inequality. Suppose that $\bbp_1, \bbp_2, \bbp_0 \in \mS$.
We want to show that
$$
d_{c,r} (\bbp_1, \bbp_2) \leq d_{c,r} (\bbp_1, \bbp_0) 
+d_{c,r} (\bbp_0, \bbp_2).
$$
Since $d_{c,r}$ is bounded above by $1$, this inequality is trivially
satisfied if either $d_{c,r}(\bbp_1, \bbp_0)=1$ or $d_{c,r} (\bbp_0,
\bbp_2)=1$. So we assume that $d_{c,r}(\bbp_1, \bbp_0)<1$ and $d_{c,r}
(\bbp_0, \bbp_2)<1$.  Let
$$
\bbp_0=\frac{\bbn_0}{\bbd_0}, \quad 
\bbp_1=\frac{\bbn_1}{\bbd_1}, \quad 
\bbp_2=\frac{\bbn_2}{\bbd_2},
$$
be coprime factorizations. If $\bbn_1^* \bbn_2+\bbd_1^*\bbd_2\in \inv
S$ and if $\iota (\bbn_1^* \bbn_2+\bbd_1^*\bbd_2)=\circ$, then the
inequality follows from the triangle inequality in $\mR^3$, that is
from the triangle inequality for the generalized chordal metric; see
\cite{Sasb}.

So we will assume that $ \neg \Big[ \bbn_1^* \bbn_2+\bbd_1^*\bbd_2\in
\inv S \textrm{ and } \iota (\bbn_1^*
\bbn_2+\bbd_1^*\bbd_2)=\circ\Big]$. Let
\begin{eqnarray*}
A
&:=& 
\underbrace {(\bbn_1^* \bbn_0 + \bbd_1^*\bbd_0)}_{\in \inv S}
\underbrace{(\bbn_0^* \bbn_2+\bbd_0^* \bbd_2)}_{\in \inv S}\in \inv S,\\
B
&:=& 
(-\bbn_1^* \bbd_0^* + \bbd_1^*\bbn_0^*)(-\bbd_0 \bbn_2+\bbn_0 \bbd_2).
\end{eqnarray*}
Suppose that $\|A^{-1}B \|_{S,\infty}<1$. Then 
$$
(\bbn_1^* \bbn_2+\bbd_1^*\bbd_2)\cdot 
\underbrace{(|\bbn_0|^2+|\bbd_0|^2)}_{\in \inv S}
=
A+B=\underbrace{A(1+A^{-1}B)}_{\in \inv S},
$$
and so $\bbn_1^* \bbn_2+\bbd_1^*\bbd_2\in \inv S$. Let 
$H:[0,1]\rightarrow \inv S$ be defined by $H(t)=A(1+tA^{-1}B)$, $t\in
[0,1]$. By the homotopic invariance of the index,
$\iota(A)=\iota(H(0))=\iota(H(1))=\iota(A+B)$. But
$\iota(A)=\circ+\circ=\circ$. And using (I5),
$\iota(A+B)=\iota(\bbn_1^* \bbn_2+\bbd_1^*\bbd_2)+\circ$. Thus $ \iota
(\bbn_1^* \bbn_2+\bbd_1^*\bbd_2)=\circ$, a contradiction.

So it can't be the case that $\|A^{-1}B\|_{S,\infty}<1$.  Since $M(S)$
is compact, it follows from here that there exists a $\varphi_0\in
M(S)$ such that $|\widehat{B}(\varphi_0)|\geq
|\widehat{A}(\varphi_0)|$. Using Lemma~\ref{lemma_cn_fact}, we obtain
that
\begin{eqnarray*}
&&
\Big(1-(\kappa_{\bbp_0,\bbp_1}(\varphi_0))^2\Big)
\Big(1-(\kappa_{\bbp_0,\bbp_2}(\varphi_0))^2\Big)
\phantom{\displaystyle\frac{\widehat{a}}{\widehat{b}}}
\\
&=&
\frac{|\widehat{A}(\varphi_0)|^2}{\left(|\widehat{\bbn}_0 (\varphi_0)|^2
 +|\widehat{\bbd}_0 (\varphi_0)|^2 \right)^2
\left(|\widehat{\bbn}_1 (\varphi_0)|^2+|\widehat{\bbd}_1 (\varphi_0)|^2\right)
\left(
|\widehat{\bbn}_2 (\varphi_0)|^2+|\widehat{\bbd}_2 (\varphi_0)|^2\right)}
\\
&\leq &
\frac{|\widehat{B}(\varphi_0)|^2}{
\left(|\widehat{\bbn}_0 (\varphi_0)|^2
+|\widehat{\bbd}_0 (\varphi_0)|^2 \right)^2 
\left(|\widehat{\bbn}_1 (\varphi_0)|^2+|\widehat{\bbd}_1 (\varphi_0)|^2\right)
\left(
|\widehat{\bbn}_2 (\varphi_0)|^2+|\widehat{\bbd}_2 (\varphi_0)|^2\right)}
\\
&=&
\Big(\kappa_{\bbp_0,\bbp_1}(\varphi_0)\Big)^2\cdot \Big((\kappa_{\bbp_0,\bbp_2}(\varphi_0)\Big)^2.
\phantom{\displaystyle\frac{\widehat{a}}{\widehat{b}}}
\end{eqnarray*}
Upon rearranging, we obtain $(\kappa_{\bbp_0,\bbp_1}(\varphi_0))^2+
(\kappa_{\bbp_0,\bbp_2}(\varphi_0))^2\geq 1$, and so
\begin{eqnarray*}
\Big(d_{c,r}(\bbp_0,\bbp_1)+d_{c,r}(\bbp_0,\bbp_2)\Big)^2 
&=&
\Big(\kappa_{\bbp_0,\bbp_1}(\varphi_0)+ \kappa_{\bbp_0,\bbp_2}(\varphi_0)\Big)^2
\\
&\geq& 
\Big(\kappa_{\bbp_0,\bbp_1}(\varphi_0)\Big)^2+
\Big(\kappa_{\bbp_0,\bbp_2}(\varphi_0)\Big)^2
\\
&\geq& 1=\Big(d_{c,r}(\bbp_1,\bbp_2)\Big)^2.
\end{eqnarray*}
This completes the proof of the triangle inequality, and also the
proof of Theorem \ref{thm_main_1}.
\end{proof}

\section{Stabilizability is a robust property}
\label{section_robustness}

\begin{proof}[Proof of Theorem~\ref{thm_main_2}] 
If $d_{c,r}(\bbp,\bbp_0)\geq \mu_{\bbp_0,\bbc}$, then trivially we have
$$
\mu_{\bbp,\bbc}\geq 0\geq \mu_{\bbp_0,\bbc}-d_{c,r}(\bbp,\bbp_0).
$$
So let us suppose that $d_{c,r}(\bbp,\bbp_0)<
\mu_{\bbp_0,\bbc}$. Since $\mu_{\bbp_0,\bbc}\leq 1$, it follows that
$d_{c,r}(\bbp,\bbp_0)<1 $. Also, $\mu_{\bbp_0,\bbc}\neq 0$, since
otherwise we would have $d_{c,r}(\bbp,\bbp_0)<0$, which is
impossible. So $\mu_{\bbp_0,\bbc}>0$. Hence $\bbc$ stabilizes
$\bbp_0$. We have
\begin{eqnarray*}
d_{c,r}(\bbp,\bbp_0)
&=& 
\sup_{\varphi \in M(S)}\frac{|\widehat{\bbn}_p(\varphi)\widehat{\bbd}_{p,0}(\varphi)
-\widehat{\bbn}_p(\varphi)\widehat{\bbd}_{p,0}(\varphi)|}{
\sqrt{|\widehat{\bbn}_{p}(\varphi)|^2+|\widehat{\bbd}_{p}(\varphi)|^2} 
\sqrt{|\widehat{\bbn}_{p,0}(\varphi)|^2+|\widehat{\bbd}_{p,0}(\varphi)|^2}}
\\
&< & 
\inf_{\varphi \in M(S)}
\frac{|\widehat{\bbn}_{p,0}(\varphi)\widehat{\bbn}_c(\varphi)-
  \widehat{\bbd}_{p,0}(\varphi)\widehat{\bbd}_c(\varphi)|}{
  \sqrt{|\widehat{\bbn}_{p,0}(\varphi)|^2+|\widehat{\bbd}_{p,0}(\varphi)|^2
  }\sqrt{
    |\widehat{\bbn}_c(\varphi)|^2+|\widehat{\bbd}_c(\varphi)|^2}}
\\
&=&\mu_{\bbp_0,\bbc}.
\phantom{\frac{\widehat{\bbn}_{p,0}(\varphi)}{\widehat{\bbn}_{p,0}(\varphi)}}
\end{eqnarray*}
So for all $\varphi \in M(S)$, 
\begin{eqnarray}
&&
\nonumber
\frac{|\widehat{\bbn}_p(\varphi)\widehat{\bbd}_{p,0}(\varphi)
-\widehat{\bbn}_p(\varphi)\widehat{\bbd}_{p,0}(\varphi)|}{
\sqrt{|\widehat{\bbn}_{p}(\varphi)|^2+|\widehat{\bbd}_{p}(\varphi)|^2} 
\sqrt{|\widehat{\bbn}_{p,0}(\varphi)|^2+|\widehat{\bbd}_{p,0}(\varphi)|^2}}
\\
\label{eq_thm_main_2_a}
&<&
 \frac{|\widehat{\bbn}_{p,0}(\varphi)\widehat{\bbn}_c(\varphi)-
  \widehat{\bbd}_{p,0}(\varphi)\widehat{\bbd}_c(\varphi)|}{
  \sqrt{|\widehat{\bbn}_{p,0}(\varphi)|^2+|\widehat{\bbd}_{p,0}(\varphi)|^2
  }\sqrt{
    |\widehat{\bbn}_c(\varphi)|^2+|\widehat{\bbd}_c(\varphi)|^2}}.
\end{eqnarray}
Using the elementary result in Lemma~\ref{lemma_cn_fact}, we obtain 
\begin{eqnarray}
\nonumber&&
\frac{|(\widehat{\bbd}_{p,0}(\varphi))^* \widehat{\bbn}_{c}(\varphi)
+(\widehat{\bbn}_{p,0}(\varphi))^*\widehat{\bbd}_{c}(\varphi)|}{
\sqrt{|\widehat{\bbn}_{c}(\varphi)|^2+|\widehat{\bbd}_{c}(\varphi)|^2} 
\sqrt{|\widehat{\bbn}_{p,0}(\varphi)|^2+|\widehat{\bbd}_{p,0}(\varphi)|^2}}
\\
\label{eq_thm_main_2_b}
&<&
 \frac{|(\widehat{\bbn}_{p,0}(\varphi))^*\widehat{\bbn}_{p}(\varphi)+
  (\widehat{\bbd}_{p,0}(\varphi))^*\widehat{\bbd}_p(\varphi)|}{
  \sqrt{|\widehat{\bbn}_{p,0}(\varphi)|^2+|\widehat{\bbd}_{p,0}(\varphi)|^2
  }\sqrt{
    |\widehat{\bbn}_p(\varphi)|^2+|\widehat{\bbd}_p(\varphi)|^2}}.
\end{eqnarray}
Since all quantities in the inequalities \eqref{eq_thm_main_2_a},
\eqref{eq_thm_main_2_b} are positive, we may multiply them to obtain
\begin{eqnarray}
\nonumber 
&&
|(\widehat{\bbd}_{p,0}(\varphi))^* \widehat{\bbn}_{c}(\varphi)
+(\widehat{\bbn}_{p,0}(\varphi))^*\widehat{\bbd}_{c}(\varphi)|
|\widehat{\bbn}_p(\varphi)\widehat{\bbd}_{p,0}(\varphi)
-\widehat{\bbn}_p(\varphi)\widehat{\bbd}_{p,0}(\varphi)|
\\
\label{eq_thm_main_2_c}
&<&
|(\widehat{\bbn}_{p,0}(\varphi))^*\widehat{\bbn}_{p}(\varphi)+
  (\widehat{\bbd}_{p,0}(\varphi))^*\widehat{\bbd}_p(\varphi)|
|\widehat{\bbn}_{p,0}(\varphi)\widehat{\bbn}_c(\varphi)-
  \widehat{\bbd}_{p,0}(\varphi)\widehat{\bbd}_c(\varphi)|
\end{eqnarray}
Define 
\begin{eqnarray*}
A&:=& 
\underbrace{(\bbn_{p,0}^* \bbn_p +\bbd_{p,0}^*\bbd_p)}_{\in 
\inv S \textrm{ as } d_{c,r}(\bbp,\bbp_0)<1}
\cdot 
\underbrace{(-\bbn_{c} \bbn_{p,0} +\bbd_{c}\bbd_{p,0})}_{\in  
\inv S\textrm{ as } \bbc \textrm{ stabilizes }\bbp_{0}}
\in \inv S,\\
B&:=& 
(-\bbn_{p} \bbd_{p,0} +\bbd_{p,0}\bbd_{p})\cdot 
(\bbd_{p,0}^* \bbn_c +\bbn_{p,0}^*\bbd_c).
\end{eqnarray*}
From \eqref{eq_thm_main_2_c}, we know that $\|A^{-1} B\|_{S,\infty}<1$. 
Since
$$
(\bbn_p \bbn_c -\bbd_p\bbd_c) \underbrace{(|\bbn_{p,0}|^2
+ |\bbd_{p,0}|^2) }_{\in \inv S} 
=A+B=\underbrace{A(1+A^{-1}B)}_{\inv S},
$$
it follows that $\bbn_p \bbn_c -\bbd_p\bbd_c\in \inv S$. Define the
map $H:[0,1]\rightarrow \inv S$ by $H(t)=A(1+tA^{-1}B)$, $t\in [0,1]$.
By Proposition~\ref{prop_hom_inv}, $\iota(H(0))=\iota(H(1))$. But
$\iota(H(0))=\iota(A)=\circ +\circ =\circ$, and 
$$
\iota(H(1))
=\iota(\bbn_p \bbn_c -\bbd_p\bbd_c)+\circ=\iota(\bbn_p \bbn_c
-\bbd_p\bbd_c).
$$
Consequently, $\iota(\bbn_p \bbn_c -\bbd_p\bbd_c)=\circ$. It follows
from Proposition~\ref{proposition_stabilization_index} that $\bbc$
stabilizes $\bbp$.

It is easy to check that
\begin{eqnarray*}
&&
(\bbn_p \bbn_c-\bbd_p \bbd_c)(|\bbn_{p,0}|^2+|\bbd_{p,0}|^2)
\\
&=&
(\bbn_{p,0}^*\bbn_p +\bbd_{p,0}^*\bbd_p)
(\bbn_c\bbn_{p,0}-\bbd_c \bbd_{p,0})+
(\bbn_p \bbd_{p,0}-\bbn_{p,0}\bbd_p)
(\bbn_c \bbd_{p,0}^*+\bbn_{p,0}^* \bbd_{c}).
\end{eqnarray*}
Thus for $\varphi \in M(S)$,
\begin{eqnarray*}
&&
\frac{|\widehat{\bbn}_p (\varphi)\widehat{\bbn}_c(\varphi)
-\widehat{\bbd}_p(\varphi) \widehat{\bbd}_c(\varphi)|}{
\sqrt{|\widehat{\bbn}_p(\varphi)|^2+|\widehat{\bbd}_p(\varphi)|^2}
\sqrt{|\widehat{\bbn}_c(\varphi)|^2+|\widehat{\bbd}_c(\varphi)|^2}}
\\
&=&
\left|
\frac{\left(\widehat{\bbn}_{p,0}^*(\varphi)\widehat{\bbn}_p(\varphi) 
+\widehat{\bbd}_{p,0}^*(\varphi)\widehat{\bbd}_p(\varphi)\right)
\left(\widehat{\bbn}_c(\varphi)\widehat{\bbn}_{p,0}(\varphi)-\widehat{\bbd}_c(\varphi) \widehat{\bbd}_{p,0}(\varphi)\right)}{
\sqrt{|\widehat{\bbn}_p(\varphi)|^2+|\widehat{\bbd}_p(\varphi)|^2}
\sqrt{|\widehat{\bbn}_c(\varphi)|^2+|\widehat{\bbd}_c(\varphi)|^2}
\left(|\widehat{\bbn}_{p,0}(\varphi)|^2+|\widehat{\bbd}_{p,0}(\varphi)|^2\right)
}
\right.
\\
&& 
+ \left. 
\frac{\left(\widehat{\bbn}_p (\varphi)\widehat{\bbd}_{p,0}(\varphi)-\widehat{\bbn}_{p,0}(\varphi)\widehat{\bbd}_p(\varphi)\right)
\left(\widehat{\bbn}_c(\varphi) \widehat{\bbd}_{p,0}^*(\varphi)+\widehat{\bbn}_{p,0}^*(\varphi) \widehat{\bbd}_{c}(\varphi)\right)
}{
\sqrt{|\widehat{\bbn}_p(\varphi)|^2+|\widehat{\bbd}_p(\varphi)|^2}
\sqrt{|\widehat{\bbn}_c(\varphi)|^2+|\widehat{\bbd}_c(\varphi)|^2}
\left(|\widehat{\bbn}_{p,0}(\varphi)|^2+|\widehat{\bbd}_{p,0}(\varphi)|^2\right)}
\right|
\\
&\geq &
(\sin \alpha )\cdot (\cos \beta)-(\cos \alpha)\cdot ( \sin \beta)=\sin (\alpha-\beta), 
\phantom{\displaystyle \frac{\left(\widehat{\bbn}_p^2\right)}{\left(\widehat{\bbn}_{p}^2\right)}}
\end{eqnarray*}
where $\alpha,\beta\in [0,\pi/2]$ are such that 
\begin{eqnarray*}
\sin \alpha 
&=& 
\frac{|-\widehat{\bbn}_{p,0}(\varphi)\widehat{\bbn}_c(\varphi) 
+\widehat{\bbd}_{p,0}(\varphi)\widehat{\bbd}_c(\varphi)|}{
\sqrt{|\widehat{\bbn}_c(\varphi)|^2+|\widehat{\bbd}_c(\varphi)|^2}
\sqrt{|\widehat{\bbn}_{p,0}(\varphi)|^2+|\widehat{\bbd}_{p,0}(\varphi)|^2}},
\\
\sin \beta
&=& 
\frac{|-\widehat{\bbn}_{p}(\varphi)\widehat{\bbd}_{p,0}(\varphi) 
+\widehat{\bbn}_{p,0}(\varphi)\widehat{\bbd}_p(\varphi)|}{
\sqrt{|\widehat{\bbn}_p(\varphi)|^2+|\widehat{\bbd}_p(\varphi)|^2}
\sqrt{|\widehat{\bbn}_{p,0}(\varphi)|^2+|\widehat{\bbd}_{p,0}(\varphi)|^2}}.
\end{eqnarray*}
Then 
\begin{eqnarray*}
\cos \alpha 
&=& 
\frac{|\widehat{\bbn}_{c}(\varphi)\widehat{\bbd}_{p,0}^*(\varphi) 
+\widehat{\bbd}_{c}(\varphi)\widehat{\bbn}_{p,0}^*(\varphi)|}{
\sqrt{|\widehat{\bbn}_c(\varphi)|^2+|\widehat{\bbd}_c(\varphi)|^2}
\sqrt{|\widehat{\bbn}_{p,0}(\varphi)|^2+|\widehat{\bbd}_{p,0}(\varphi)|^2}},
\\
\cos \beta
&=& 
\frac{|\widehat{\bbn}_{p}(\varphi)\widehat{\bbn}_{p,0}^*(\varphi) 
+\widehat{\bbd}_{p,0}(\varphi)\widehat{\bbd}_{p,0}^*(\varphi)|}{
\sqrt{|\widehat{\bbn}_p(\varphi)|^2+|\widehat{\bbd}_p(\varphi)|^2}
\sqrt{|\widehat{\bbn}_{p,0}(\varphi)|^2+|\widehat{\bbd}_{p,0}(\varphi)|^2}}.
\end{eqnarray*}
Since $\sin^{-1}:[0,1]\rightarrow [0,\pi/2]$ is increasing, we obtain for $\varphi \in M(S)$,
$$
\sin^{-1} 
\frac{|\widehat{\bbn}_p (\varphi)\widehat{\bbn}_c(\varphi)
-\widehat{\bbd}_p(\varphi) \widehat{\bbd}_c(\varphi)|}{
\sqrt{|\widehat{\bbn}_p(\varphi)|^2+|\widehat{\bbd}_p(\varphi)|^2}
\sqrt{|\widehat{\bbn}_c(\varphi)|^2+|\widehat{\bbd}_c(\varphi)|^2}}
\geq  
\alpha -\beta.
$$
But 
\begin{eqnarray*}
\alpha-\beta
&=&
\sin^{-1} 
\frac{|-\widehat{\bbn}_{p,0}(\varphi)\widehat{\bbn}_c(\varphi) 
+\widehat{\bbd}_{p,0}(\varphi)\widehat{\bbd}_c(\varphi)|}{
\sqrt{|\widehat{\bbn}_c(\varphi)|^2+|\widehat{\bbd}_c(\varphi)|^2}
\sqrt{|\widehat{\bbn}_{p,0}(\varphi)|^2+|\widehat{\bbd}_{p,0}(\varphi)|^2}}
\\
&& 
-
\sin^{-1} 
\frac{|-\widehat{\bbn}_{p}(\varphi)\widehat{\bbd}_{p,0}(\varphi) 
+\widehat{\bbn}_{p,0}(\varphi)\widehat{\bbd}_p(\varphi)|}{
\sqrt{|\widehat{\bbn}_p(\varphi)|^2+|\widehat{\bbd}_p(\varphi)|^2}
\sqrt{|\widehat{\bbn}_{p,0}(\varphi)|^2+|\widehat{\bbd}_{p,0}(\varphi)|^2}}
\\
&\geq & 
\sin^{-1} \mu_{\bbp_0,\bbc}
-
\sin^{-1} 
\frac{|-\widehat{\bbn}_{p}(\varphi)\widehat{\bbd}_{p,0}(\varphi) 
+\widehat{\bbn}_{p,0}(\varphi)\widehat{\bbd}_p(\varphi)|}{
\sqrt{|\widehat{\bbn}_p(\varphi)|^2+|\widehat{\bbd}_p(\varphi)|^2}
\sqrt{|\widehat{\bbn}_{p,0}(\varphi)|^2+|\widehat{\bbd}_{p,0}(\varphi)|^2}}
\\
&\geq & 
\sin^{-1} \mu_{\bbp_0,\bbc}
-
\sin^{-1} d_{c,r}(\bbp,\bbp_0).
\phantom{\displaystyle \frac{\left(\widehat{\bbn}_p^2\right)}{\left(\widehat{\bbn}_{p}^2\right)}}
\end{eqnarray*}
Hence 
\begin{equation}
\label{eq_thm_main_2_d}
\sin^{-1} \mu_{\bbp,\bbc}\geq \sin^{-1} \mu_{\bbp_0,\bbc}-\sin^{-1} d_{c,r}(\bbp,\bbp_0). 
\end{equation}
For $x,y,z\in [0,1]$, if $ \sin^{-1} x\leq \sin^{-1} y+\sin^{-1} z  $, then 
 by taking the cosine of both sides and using that  $\cos$ is a
decreasing function on $[0, \frac{\pi}{2}]$, we  get $
\sqrt{1-x^2} \geq \sqrt{1-y^2} \sqrt{1-z^2}-yz$, which in turn implies
that $(\sqrt{1-x^2}+yz)^2\geq (1-y^2)(1-z^2)$.  Hence 
$$ 
x^2 \leq y^2
+z^2+2yz\sqrt{1-x^2}\leq y^2 +z^2+2yz\cdot 1=(y+z)^2,
$$
which gives finally that $ x\leq y+z$.  The claimed inequality now
follows immediately from the inequality in \eqref{eq_thm_main_2_d}
upon setting $x = \mu_{\bbp_0,\bbc}$, $y = d_{c,r}(\bbp_0,\bbp)$ and
$z = \mu_{\bbp,\bbc}$.
\end{proof}

\section{Specific instances of $R$}

Table~\ref{table} below gives an overview of the choice of 
principal objects $S,G,\iota$ specific to choices of $R$ as the 
standard classes of stable transfer functions used in control theory. 

\begin{table}[!hbp]
\begin{center}
\begin{tabular}{|c||c|c|c|}\hline 
$R\phantom{\Big(}$ & $S$ & $G$ & $\iota$ \\ \hline \hline 
$\begin{array}{c}
{RH^\infty}^{\phantom{\Big(}}\\
A(\mD),\\
W^+(\mD),\\
\widehat{L^1[0,\infty)}+\mC,\\
\cdots 
\end{array}$ & $C(\mT)$ & $\mZ$ & $f\mapsto w(f)$ \\ \hline
$\calA_+ {\phantom{\Bigg(}}$ & $C_0+AP$ & $\mR\times \mZ$ & $f=f_0+f_{AP}\mapsto 
\Big( w_{\textrm{av}}(f_{AP})+ w(1+f_{AP}^{-1} \widehat{f}_{0})\Big)$
\\
\hline  
$H^\infty$ & $\begin{array}{ccc}
\underrightarrow{\lim}\;C_{b}(\mathbb{A}_{r})^{\phantom{\Big(}}\\
\textrm{\rotatebox[origin=c]{270}{$\simeq$}}\\
C(\beta \mA_0 \setminus \mA_0)_{\phantom{\Big(}}
\end{array}$  & $\mZ$ & $[(f_r)_r] \mapsto \displaystyle \lim_{r\rightarrow 1} w(f_r)$ \\ \hline 
\end{tabular}
\label{table}

$\;$

\caption{Choices of $S, G,\iota$ corresponding to specific instances of $R$.}
\end{center}
\end{table}

\subsection{$R=RH^\infty, A(\mD), W^+(\mD), \widehat{L^1[0,\infty)}+\mC,\cdots$}

Let \label{def_subsets_of_C} 
\begin{eqnarray*}
\mC_{\scriptscriptstyle >0}&:=&\{s\in \mC:\textrm{Re}(s)>0\},\\
\mD&:=& \{z\in \mC:|z|<1\},\\
\mT&:=& \{z\in \mC:|z|=1\},\\
\overline{\mD}&:=& \mD\cup \mT.
\end{eqnarray*}
Recall that \label{def_classical_algebras}
\begin{eqnarray*}
 RH^\infty&:=&\{ f:\mC_{\scriptscriptstyle >0}\rightarrow \mC: 
 f \textrm{ is rational and  bounded in }\mC_{\scriptscriptstyle >0}\},
 \phantom{\sum_{n=0}^\infty }\\
 A(\mD)&:=&\{f:\overline{\mD} \rightarrow \mC: 
 f\textrm{ is holomorphic in }\mD\textrm{ and continuous in }
 \overline{\mD}\},\phantom{\sum_{n=0}^\infty }\\
 W^+(\mD)&:=&\{f: f(z)=\sum_{n=0}^\infty a_n z^n \;(z\in \overline{\mD})\textrm{ and }\sum_{n=0}^\infty |a_n|<+\infty\},\\
 \widehat{L^1[0,\infty)}+\mC&:=& \{ f: f(s)=\widehat{f}_{\textrm{a}}(s)+ f_0 \;(s\in \mC_{\scriptscriptstyle >0}), 
 \textrm{ where } f_{\textrm{a}}\in L^1[0,\infty), \; f_0\in \mC\}.\phantom{\sum_{n=0}^\infty }
\end{eqnarray*}
In the above, $\widehat{f}_{\textrm{a}}$ \label{def_Laplace} denotes the 
Laplace transform of $f_{\textrm{a}}\in L^1[0,\infty)$. We remark that 
the set of bounded and holomorphic functions defined in the open right half plane, and which   
possess a continuous extension to $i\mR \cup \{\infty_{\scriptscriptstyle \mC}\}$,  
can be transplanted to functions on the closed unit disk using the conformal map given by
$$
\varphi(z)=\frac{1+z}{1-z} \quad (z\in {\mathbb{D}}).
$$
In this manner we may think of elements of $\widehat{L^1[0,\infty)}+\mC$ as being elements of 
the disc algebra.

In each of these cases, the values of the function on the boundary of the domain of 
definition gives rise to a function which can be considered to be an element of the 
$C^*$-algebra \label{def_C(T)}
$$
S:=C(\mT):=\{f:\mT\rightarrow \mC: f\textrm{ is continuous on }\mT\}.
$$
We take $G=\mZ$, and $\iota:\inv C(\mT)\rightarrow \mZ$ to be the winding number $w$ 
with respect to the origin:\label{def_winding_number}
$$
\iota (f):=w(f),\quad f\in \inv C(\mT).
$$
Then it can be checked that (A1)-(A3) hold. Most of the details can be found in \cite{BalSas}, 
except for  (I5), but this verification is obvious. 
Moreover,  the $\|\cdot\|_{S,\infty}$-norm in the definition of the 
$d_{c,r}$-metric is the usual supremum $\|\cdot\|_\infty$-norm of functions in $C(\mT)$. 

\subsection{$R=\calA_+$} 

Recall that \label{def_Callier_Desoer}
$$
\calA^+=\left\{ s (\in \mC_{\scriptscriptstyle >0}) \mapsto \widehat{f_{\textrm{a}}}(s)
  +\displaystyle \sum_{k=0}^\infty a_k e^{- s t_k} \; \bigg| \;
\begin{array}{ll}
f_{\textrm{a}} \in L^{1}[0,\infty), \;(a_k)_{k\geq 0} \in \ell^{1},\\
0=t_0 <t_1 ,t_2 , t_3, \dots
\end{array} \right\}
$$
Let \label{def_C_0}
\begin{eqnarray*}
 C_0&:=&\{f:\mR\rightarrow \mC: f\textrm{ is continuous on }\mR\textrm{ and }\lim_{x\rightarrow\pm\infty }f(x)=0\}\\
 AP&:=&\textrm{closed span in }L^\infty(\mR)\textrm{ of } \{(\mR\owns)x\mapsto e^{i\lambda x}:\lambda\in \mR\}.
\end{eqnarray*}
$C_0+AP$, endowed with pointwise operations, with the supremum norm, and with involution given by 
pointwise complex conjugation,  is a sub-$C^*$-algebra of $L^\infty(\mR)$; \cite{Dou}. 
We will take $S:=C_0+AP$. Moreover, 
we take $G=\mR\times \mZ$, and define $\iota:\inv (C_0+AP)\rightarrow \mR\times \mZ$ by 
$$
\iota(f)=\Big( w_{\textrm{av}}(f_{AP}), \; w(1+f_{AP}^{-1} f_0) \Big), \quad f=f_0+ f_{AP}\in \inv( C_0+AP), 
\;f_0\in C_0,\;f_{AP}\in AP.
$$
In the above, $w_{\textrm{av}}:\inv AP\rightarrow \mR$ \label{def_average_winding} denotes the average winding number, defined by 
$$
w_{\textrm{av}}(f_{AP}):=\lim_{x \rightarrow +\infty} \frac{\arg (f(x))-\arg(f(-x))}{2x},\quad f_{AP}\in \inv AP,
$$
see \cite[Theorem 1, p. 167]{JesTor}. Again, it can be checked that (A1)-(A3) hold; see \cite{BalSas}. 
Since $C_0+AP$ is a sub-$C^*$-algebra of $L^\infty(\mR)$, the $\|\cdot\|_{S,\infty}$-norm in the definition of the 
$d_{c,r}$-metric is the usual $\|\cdot\|_\infty$-norm of functions in $L^\infty(\mR)$. 

\subsection{$R=H^\infty$}

The Hardy algebra $H^{\infty}$ consists of all bounded and holomorphic
functions defined on the open unit disk $\mathbb{D}:=\{z\in\mathbb{C}
:|z|<1\}$, with pointwise operations and the usual supremum norm. 
We recall the construction of $S$ from \cite{FreSas}. For given $r\in(0,1),$ let\label{def_annulus}
$$
\mathbb{A}_{r}:=\{z\in\mathbb{C}:r<|z|<1\}
$$
denote the open annulus. In particular, 
$$
\mathbb{A}_{0}:=\{z\in\mathbb{C}:0<|z|<1\}
$$
is the punctured disc. 
Let $C_{b}(\mathbb{A}_{r})$ be the $C^{\ast}
$-algebra of all bounded and continuous functions $f:\mathbb{A}_{r}
\rightarrow\mathbb{C}$, equipped with pointwise operations and the supremum
norm. Moreover, for $0<r\leq R<1$ we define the map $\pi_{r}^{R}:C_{b}
(\mathbb{A}_{r})\rightarrow C_{b}(\mathbb{A}_{R})$ by restriction:\label{def_pi_r^R}
$$
\pi_{r}^{R}(f)=f|_{\mathbb{A}_{R}}, \quad f\in C_{b} (\mathbb{A}_{r}).
$$
Consider the family $\left(  C_{b}(\mathbb{A}_{r}),\pi_{r}^{R}\right)  $ \label{def_C_b} for
$0<r\leq R<1$. We note that
\begin{itemize}
\item[(i)] $\pi_{r}^{r}$ is the identity map on $C_{b}(\mathbb{A}_{r})$, and

\item[(ii)] $\pi_{r}^{R}\circ\pi_{\rho}^{r}=\pi_{\rho}^{R}$ for all
$0<\rho\leq r\leq R<1.$
\end{itemize}
Now consider the $\ast$-algebra\label{def_Pi_C_b}
\[
{\displaystyle\prod\limits_{r\in(0,1)}} C_{b}(\mathbb{A}_{r}),
\]
and denote by $\mathcal{A}$\label{def_calA} its $\ast$-subalgebra consisting of all elements
$f=(f_{r})=(f_{r})_{r\in(0,1)}$ such that there is an index $r_{0}$ with
$\pi_{r}^{R}(f_{r})=f_{R}$ for all $0<r_{0}\leq r\leq R<1$. Since every
$\pi_{r}^{R}$ is norm decreasing, the net $(\|f_{r}\|_{\infty})$ is convergent and we define
\[
\|f\|:=\lim_{r\rightarrow1}\|f_{r}\|_{\infty}.
\]
Clearly this defines a seminorm on $\mathcal{A}$ that satisfies the $C^{\ast}
$-norm identity, that is,
\[
\|f^{\ast}f\|=\|f\|^{2},
\]
where $\cdot^{\ast}$ is the involution, that is, complex conjugation, see
(\ref{involution}) below. Now, if $N$ \label{def_N} is the kernel of $\|\cdot\|$, then the
quotient $\mathcal{A}/N$ is a $C^{\ast}$-algebra (and we denote the norm again
by $\|\cdot\|$). This algebra is the  \emph{direct}/\emph{inductive limit} 
of $(C_{b}(\mathbb{A}_{r}),\pi_{r}^{R})$ and we denote
it by\label{def_dir_lim}
\[
\underrightarrow{\lim}\;C_{b}(\mathbb{A}_{r}).
\]
To every element $f\in C_{b}(\mathbb{A}_{r_{0}}),$ we associate a sequence
$f_{1}=(f_{r})$ in $\mathcal{A}$, where
\begin{equation}
f_{r}=\left\{
\begin{array}
[c]{ll}
0 & \text{if }0<r<r_{0},\\
\pi_{r_{0}}^{r}(f) & \text{if }r_{0}\leq r<1.
\end{array}
\right. \label{f1}%
\end{equation}
We also define a map $\pi_{r}:C_{b}(\mathbb{A}_{r})\rightarrow
\underrightarrow{\lim}\;C_{b}(\mathbb{A}_{r})$ by\label{def_pi_r}
$$
\pi_{r}(f):=[f_{1}],\quad f\in C_{b}(\mathbb{A}_{r}),
$$
where $[f_{1}]$ denotes the equivalence class in $\underrightarrow{\lim
}\;C_{b}(\mathbb{A}_{r})$ which contains $f_{1}$. We will use the fact that
the maps $\pi_{r}$ are in fact $\ast$-homomorphisms. We note that these maps
are compatible with the connecting maps $\pi_{r}^{R}$ in the sense that every
diagram shown below is commutative.
$$
\xymatrix{
C_b(\mA_r) \ar[r]^{\pi_r^R} \ar[rd]_{\pi_r} & C_b(\mA_R) \ar[d]^{\pi_R} \\ &
\displaystyle {\lim_{\longrightarrow} C_b(\mA_r)}
}
$$
Then $\underrightarrow{\lim}\;C_{b}(\mathbb{A}_{r})$ is a $C^\ast$-algebra, 
see \cite[Section~2.6]{LOT99}. The multiplicative identity arises from the constant function
$f\equiv1 $ in $C_{b}(\mathbb{A}_{0})$, that is, $\pi_{0}(f)$. Moreover, we
can define an involution in $C_{b}(\mathbb{A}_{r})$ by setting
\begin{equation}
(f^{\ast})(z):=\overline{f(z)},\ \ \ z\in\mathbb{A}_{r},\label{involution}%
\end{equation}
and this implicitly defines an involution of elements in
$\underrightarrow{\lim}\; C_{b}(\mathbb{A}_{r})$. 
There is a natural embedding of $H^{\infty}$ into
$\underrightarrow{\lim}\; C_{b}(\mathbb{A}_{r})$, namely
\begin{equation}
f\mapsto\pi_{0}(f):H^{\infty}\longrightarrow\underrightarrow{\lim}\; C_{b}
(\mathbb{A}_{r}).\label{injektiv}
\end{equation}
We will take $G=\mZ$. For $f\in{\text{inv }}
(C_{b}(\mathbb{A}_{\rho}))$ and for $0<\rho<r<1$ we define the map
$f^{r}:\mathbb{T}\rightarrow\mathbb{C}$ by $f^{r}(\zeta)=f(r\zeta)$, 
$\zeta\in\mathbb{T}$. If $f\in{\text{inv }} (C_{b}(\mathbb{A}_{\rho}))$, then $f^{r}\in{\text{inv }}
(C(\mathbb{T}))$, and so $f^{r}$ has a winding number $w(f_r)$. 
We set \label{def_w} $w(f):=w(f^{r})\in\mathbb{Z}$ with respect to $0$, and it can be shown that this is well-defined. 
Now we define the map
$\iota:{\text{inv }} (\underrightarrow{\lim}\;C_{b}(\mathbb{A}_{r}))\rightarrow
\mathbb{Z}$. For  $[(f_{r})]\in{\text{inv }} (\underrightarrow{\lim}\ C_{b}(\mathbb{A}_{r}
))$, 
\begin{equation}
\iota(f):=\lim_{r\rightarrow1}w(f_{r}),\text{ for }
f=[(f_{r})]\in{\text{inv }} (\underrightarrow{\lim}\;C_{b}(\mathbb{A}
_{r})).
\end{equation}
It can be shown that $\iota$ is well-defined and all the properties we demand are satisfied; see \cite{FreSas}. 
It was also shown there that 
$\underrightarrow{\lim}\;C_{b}(\mathbb{A}_{r}) $ is isometrically isomorphic to  
$ C(\beta\mA_0 \setminus \mA_0)$ (where \label{def_stone}
$$
\beta \mA_0:=\textrm{Stone-\v{C}ech compactification of }\mA_0,
$$
which is the maximal ideal space of the Banach algebra $C_b(\mA_0)$ 
of all complex-valued bounded continuous functions on $\mA_0$), and moreover 
$\underrightarrow{\lim}\;C_{b}(\mathbb{A}_{r}) $ is a sub-$C^\ast$-algebra of $L^\infty(\mT)$.
From here we see that the $\|\cdot\|_{S,\infty}$-norm in the definition of the 
$d_{c,r}$-metric is the usual $\|\cdot\|_\infty$-norm of functions in $L^\infty(\mR)$. 

\section{An example}

As an illustration, we consider \cite[Example~4.3]{PS}, namely 
$$
\bbp_1=\frac{1}{s-e^{-s}}.
$$
It was shown in \cite{PS} that an algebraic spectral factorization leading 
to normalized coprime factors is not possible in this example. 
So calculating the $\nu$-metric between this plant and a perturbed one, say 
$$
\bbp_a:=\frac{1}{s-ae^{-s}},
$$
is problematic. Nevertheless, there exists a coprime factorization of
$$
\bbp_a=\displaystyle\frac{\bbn_a}{\bbd_a}
$$
over $H^\infty$ of the half-plane 
$\mC_{\scriptscriptstyle >0}$, where 
$$
\bbn_a:=\frac{1}{1+s},\quad \bbd_a:=\frac{s-ae^{-s}}{1+s}.
$$
That this is a coprime factorization over $H^\infty$ follows from the fact that 
\begin{eqnarray*}
 \bbx&:=&1+e^{-s}, \\
 \bby&:=& 1
\end{eqnarray*}
solve the Bezout equation $\bbn \bbx + \bbd \bby=1$:
$$
\bbn \bbx + \bbd \bby=
\frac{1}{1+s} \cdot (1+e^{-s}) + \frac{s-e^{-s}}{1+s}\cdot 1=1.
$$
We will see that using this coprime factorization, we can compute
$d_{c,r}(\bbp_a,\bbp_1)$ fairly easily. We remark also that $\bbp_a$ is
the transfer function associated with the retarded delay differential
equation
$$
\left.\begin{array}{rcl}
x'(t)&=&a\cdot x(t-1)+u(t),\\
y(t)&=& x(t)\end{array}\right\}, \quad t\geq 0.
$$
For $s\in \mC_{\scriptscriptstyle > 0}$, 
and with $a=:1+\delta$, we have 
$$
\overline{\bbn_1(s)}\cdot \bbd_a(s)+\overline{\bbd_1(s)}\cdot \bbn_a(s)
=
\underbrace{\frac{1}{|1+s|^2} + \frac{|s-e^{-s}|^2}{|1+s|^2}}_{=:f_1(s)} -\delta \cdot 
\underbrace{\frac{e^{-s}(\overline{s}-e^{-\overline{s}}) }{|1+s|^2}}_{=:f_2(s)}.
$$
It is easy to see that since 
 $
\displaystyle \lim_{\substack{s\in \mC_{\scriptscriptstyle >0}\\ s\rightarrow \infty}} f_1(s)=1,
$ 
and  for all $s\in \mC_{\scriptscriptstyle >0}$,
 $
f_1(s)\geq \displaystyle \frac{1}{|1+s|^2}>0,
$ we have
$$
m:=\inf_{s\in \mC_{\scriptscriptstyle >0}} f_1(s)>0.
$$
Also, 
 $\displaystyle 
\lim_{\substack{s\in \mC_{\scriptscriptstyle >0}\\ s\rightarrow \infty}} f_2(s)=0,
$ 
and so 
 $$\displaystyle 
M:= \sup_{s\in \mC_{\scriptscriptstyle >0}} | f_2(s)|<+\infty.
 $$
Consequently for all $\delta=a-1$ small enough, for example, $|\delta|=|a-1|<\displaystyle \frac{m}{2(M+1)}$, we have
$$
\textrm{Re}\Big(\overline{\bbn_1(s)}\cdot \bbd_a(s)+\overline{\bbd_1(s)}\cdot \bbn_a(s)\Big)
>m-|\epsilon|M>m-\frac{m}{2(M+1)} \cdot M>m-\frac{m}{2}=\frac{m}{2}>0.
$$
So the condition that $\bbn_1^*\bbd_a+\bbd_1^*\bbn_a\in \inv S$ and $\iota(\bbn_1^*\bbd_a+\bbd_1^*\bbn_a)=\circ$ will be 
satisfied for all such values of the parameter $a$. Hence 
\begin{eqnarray*}
d_{c,r}(\bbp_1,\bbp_a)&=&\sup_{y\in \mR} \frac{|\bbn_1(iy) \bbd_a(iy)-\bbd_1(iy) \bbn_a(iy)|}{
\sqrt{|\bbn_1(iy)|^2+|\bbd_1(iy)|^2}\sqrt{|\bbn_a(iy)|^2+|\bbd_a(iy)|^2}}
\\
&=& 
\sup_{y\in \mR} \frac{\displaystyle \frac{|a-1|}{1+y^2} }{
\sqrt{\displaystyle \frac{1+y^2+1+2y\sin y }{1+y^2} }\sqrt{ \displaystyle \frac{1+y^2+a^2+2ay\sin y }{1+y^2}} }
\\
&=& \frac{|a-1|}{\sqrt{2(1+a^2)}},
\end{eqnarray*}
for all $a\in \mR $ satisfying 
$$
|a-1|<\displaystyle \frac{m}{2(M+1)},
$$
where $m,M$ are given by 
\begin{eqnarray*}
m&:=& \inf_{s\in \mC_{\scriptscriptstyle >0}} \left(\frac{1}{|1+s|^2} + \frac{|s-e^{-s}|^2}{|1+s|^2}\right),\\
M&:=&\sup_{s\in \mC_{\scriptscriptstyle >0}} \frac{|e^{-s}(\overline{s}-e^{-\overline{s}}) |}{|1+s|^2}.
\end{eqnarray*}
It is clear that $M\geq 0$ and 
$$
m\leq 2=\left(\frac{1}{|1+s|^2} + \frac{|s-e^{-s}|^2}{|1+s|^2}\right)\bigg|_{s=0} ,
$$
so that 
$$
d_{c,r}(\bbp_1,\bbp_a)=\frac{|a-1|}{\sqrt{2(1+a^2)}}
$$
for all $a\in \mR$ satisfying $|a-1|<1$.

We remark that a stabilizing controller for the nominal plant $\bbp_1$ is trivially found to be 
$$
\bbc=-\frac{\bbx}{\bby},
$$
where $\bbx:=1+e^{-s}$ and $\bby:= 1$ (which 
solve the Bezout equation $\bbn \bbx + \bbd \bby=1$). 
The stability margin of the pair $(\bbp_1,\bbc)$ is given by 
\begin{eqnarray*}
\mu_{\bbp_1,\bbc}^{-1}&=&\|H(\bbp_1,\bbc)\|_\infty=
\left\| \left[ \begin{array}{cc}\displaystyle \frac{1+e^{-s}}{1+s} & \displaystyle \frac{1}{1+s}\\
                                       \displaystyle \frac{(1+e^{-s})(s-e^{-s})}{1+s}& \displaystyle \frac{s-e^{-s}}{1+s}
\end{array}\right]\right\|_\infty
\\
&=&
\left\| \left[ \begin{array}{cc}\displaystyle  \frac{1}{1+s} \\ \displaystyle \frac{s-e^{-s}}{1+s}\end{array}\right]\left[ \begin{array}{cc}
                                       1+e^{-s} & 1
\end{array}\right]\right\|_\infty.
\end{eqnarray*}
It follows that 
$$
\mu_{\bbp_1,\bbc}^{-1}=
\sup_{s\in i\mR} \sqrt{\bigg(\Big|\frac{1}{1+s} \Big|^2+\Big|\frac{s-e^{-s}}{1+s}\Big|^2\bigg)\Big( |1+e^{-s}|^2+1^2\Big)}.
$$
A crude upper bound for $\mu_{\bbp_1,\bbc}^{-1}$ can be shown to be $5$ (a Maple plot reveals the value being around 3.224):
\begin{eqnarray*}
\mu_{\bbp_1,\bbc}^{-1}&=&\sup_{y\in \mR} \sqrt{\frac{2+2y\sin y+y^2}{1+y^2}(3+2\cos y)}
\leq \sup_{y\in \mR} \sqrt{\frac{2+2y^2+y^2}{1+y^2}(3+2\cdot 1)}\\
&\leq& \sup_{y\in \mR} \sqrt{2\cdot \frac{1}{1+y^2} + 3\cdot \frac{y^2}{1+y^2}}\sqrt{5}
\leq \sqrt{2\cdot 1 + 3\cdot 1}\sqrt{5}=5.
\end{eqnarray*}
Thus  $\mu_{\bbp_1,\bbc}\geq 1/5$. From Theorem~\ref{thm_main_2}, it follows that for all 
$a\in \mR$ such that $|a-1|<1$, and with 
$$
d_{c,r} (\bbp, \bbp_1)=\frac{|a-1|}{\sqrt{2(1+a^2)}}<\frac{1}{5},
$$
$\bbp$ is stabilized by $\bbc$, guaranteeing a stability margin of 
$$
\frac{1}{5}-\frac{|a-1|}{\sqrt{2(1+a^2)}}>0.
$$
We remark that the inequality
$$
\frac{|a-1|}{\sqrt{2(1+a^2)}}<\frac{1}{5}
$$ 
is satisfied by all $a\in \mR$ in the interval $(2/3, 3/2)$, 
and then also $|a-1|<1$. So all plants 
$$
\bbp= \frac{1}{s-ae^{-s}}
$$ 
with $a\in (2/3, 3/2)$ are stabilized by the controller $\bbc=-(1+e^{-s})$ which stabilizes the nominal plant 
$$
\bbp_1=\frac{1}{s-e^{-s}}.
$$

\medskip 

\noindent {\bf Acknowledgements:} The author thanks the two referees for their comments which greatly improved the 
presentation of the article.

\end{document}